\documentclass[11pt]{ip-journal}
\usepackage{amscd}
\usepackage{cases}
\usepackage{amsmath}
\usepackage{amssymb, pb-diagram,slashed,tikz-cd}
\usepackage{amscd,amssymb,color,url}
\usepackage{tikz}

\newtheorem{theorem}[subsection]{Theorem}
\newtheorem{lemma}[subsection]{Lemma}

\newtheorem{corollary}[subsection]{Corollary}

\theoremstyle{definition}
\newtheorem{remark}[subsection]{Remark}
\newtheorem{definition}[subsection]{Definition}

\numberwithin{equation}{section}

\renewcommand{\theenumi}{\roman{enumi}}
\renewcommand{\labelenumi}{(\theenumi)}

\newcommand{\R}{{\mathbb R}}

\newcommand{\G}{{\Gamma}}
\newcommand{\bs}{{\backslash}}
\newcommand{\wt}[1]{{\widetilde{#1}}}

\usepackage[mathscr]{eucal}

\usepackage{xy}
\xyoption{all}

\DeclareMathOperator{\aff}{Aff}

\DeclareMathOperator{\aut}{{Aut}}

\DeclareMathOperator{\diam}{{diam}}

\DeclareMathOperator{\syl}{{Syl}}

\DeclareMathOperator{\vol}{{vol}}

\begin{document}

\title[Almost Flat Manifolds]
{An almost flat manifold with a cyclic or quaternionic holonomy group bounds}

\author{James F. Davis}
\thanks{The first author is supported by a NSF Grant and the second author is supported a NSFC Key Grant.  The second author wishes to thank Weiping Zhang for bring this problem to his attention.}
\address{Department of Mathematics, Indiana University, Bloomington,
IN 47405, USA}
\email{jfdavis@indiana.edu}

\author{Fuquan Fang}
\address{Department of Mathematics, Capital Normal
University, Beijing, 100048, China} \email{fuquan$\_$fang@yahoo.com}

\begin{abstract}
A long-standing conjecture of Farrell and Zdravkovska and independently S.~T.~Yau states that every almost flat manifold is the boundary of a compact manifold.   This paper gives a simple proof of this conjecture when the holonomy group is cyclic or quaternionic.   The proof is based on the interaction between flat bundles and involutions.
\end{abstract}

\maketitle

\section{Introduction}
\label{sec:intro}

A closed manifold $M$ is   {\it almost flat} if there is
a sequence of metrics $g_i$ on $M$ so that $\vert K_{g_i} \vert \diam
(M,{g_i})^2\to 0$  when $i\to \infty$, where $K_{g_i}$ is the
sectional curvature and $\diam(M,{g_i})$ is the diameter of $M$ with
respect to the metric $g_i$. In his celebrated paper
\cite{Gr1}, Gromov generalized the classical Bieberbach theorem for flat manifolds and proved that every almost flat manifold is finitely covered by a  {\em
nilmanifold},  that is, the quotient of a simply connected nilpotent Lie group by a uniform lattice. (Conversely, work of Farrell and Hsiang \cite{FH} showed that every manifold finitely covered by a nilmanifold is homeomorphic to an almost flat manifold.)
Ruh \cite{Ruh} strengthened Gromov's theorem and proved that an almost flat manifold is diffeomorphic to an {\em infranilmanifold},  that is,  a double coset space $\G\backslash L \rtimes \aut(L) / \aut(L)$ where $L$ is a simply connected nilpotent Lie group and $\G$ is a torsion-free subgroup of the affine group $L \rtimes \aut(L)$ so that the kernel of $\G \to \aut(L)$ has finite index in $\G$ and is discrete and cocompact in $L$.   In fact, Ruh produced a flat connection with parallel torsion on the tangent bundle of an almost flat manifold.  The map $\Gamma \to \aut(L)$  is the holonomy of this connection.  Conversely, it is not difficult to see that every infranilmanifold is almost flat. The class of almost flat
manifolds is much larger than flat manifolds; there are infinitely
many almost flat manifolds in every dimension greater than two.

Almost flat manifolds play a fundamental role in  Riemannian
geometry. By the profound Cheeger-Fukaya-Gromov theorem \cite{CFG},
almost flat manifolds are the fibers in a collapsing sequence of
Riemannian manifolds with bounded curvature and diameter.  All
cuspidal ends of a complete Riemannian manifold with finite
volume and negative pinched sectional curvature are almost flat
manifolds, generalizing the fact that the cuspidal ends of a finite volume
hyperbolic manifold  are flat manifolds.   Pedro Ontaneda recently proved an amazing converse (see Theorem A of \cite{Ont}):   If an almost flat manifold $M$ is the boundary of a compact manifold, then there is a compact manifold $X$ with boundary $M$ so that $X - M$ admits a complete, finite volume Riemannian metric with negative pinched sectional curvature.

By rescaling,  one sees that for an almost flat manifold $M$,  there
is a sequence of Riemannian metrics $g_i$ on $M$ so that
$\diam(M,{g_i})=1$ for all $i$ and  $|K_{g_i}|\to 0$ as $i \to
\infty$.    By Chern-Weil theory, the Pontryagin numbers of an
oriented closed manifold  are integrals of the Pfaffin forms on the
curvature form, and for an almost flat manifold these integrals must
converge to zero as $i \to \infty$, since by the volume comparison
theorem the sequence $\vol(M,{g_i})$ is bounded above.
 Therefore the Pontryagin numbers of an oriented almost flat manifold $M$ all vanish.
 It follows that the disjoint union of $M$ with itself is
an oriented boundary.  Furthermore, if $M$ has an almost complex
structure then, by the same reasoning, all the Chern numbers of $M$
vanish. In particular, this implies that $M$ bounds.  (Throughout this
paper when we say $M$ bounds we mean that $M$ is diffeomorphic to
the boundary of a compact manifold.) This clearly suggests a natural
and very interesting conjecture, {\it almost flat manifolds are
boundaries,} due to Farrell and Zdravkovska \cite{FZ}, which is 
posed independently in the famous problem list of S.T.Yau \cite{Ya}. It is
a well-known theorem of Thom that a closed manifold bounds if and only if all its
Stiefel-Whitney numbers vanish. Wall showed that a closed oriented manifold bounds an
orientable manifold if and only if all Stiefel-Whitney numbers and
all Pontryagin numbers are zero. The above discussion implies that
if an oriented almost flat manifold bounds, then it bounds
orientably.

A remarkable
theorem of Hamrick and Royster \cite{HR} shows that every flat
manifold bounds.  But the corresponding statement for almost flat manifolds remains a conjecture.  In some special cases it has been proven.    The {\em holonomy group} of an infranilmanifold is the
finite group $G$ given by the image of the fundamental group $\G$ in $\aut(L)$.  Farrell and Zdravkovska \cite{FZ} proved that
almost flat manifolds bound provided either that the holonomy group $G$ has order two or that  the holonomy group $G$ acts effectively on the center of $L$.    Upadhyay \cite{Upad} proved that an almost flat manifold bounds if all of the following conditions hold:  $G$ is cyclic, $G$ acts trivially on the center of $L$, and $L$ is 2-step nilpotent.

This paper contains a new and quite simple proof of the above results and proves the more general statement:

\begin{theorem} \label{main_theorem}
Let $M$ be an almost flat manifold and let $\syl_2G$ be the 2-sylow subgroup of its holonomy group.  If $\syl_2G$ is cyclic or generalized quaternionic, then $M$ is the boundary of a compact manifold.
\end{theorem}

Since all rational Pontryagin numbers vanish, from the above
theorem and cobordism theory it follows that every oriented almost
flat manifold with such a holonomy group bounds an oriented
manifold.  However, it remains difficult to answer:

\vskip 2mm

\noindent {\bf Problem} (a). Does every almost flat Spin manifold
(up to changing Spin structures) bound a Spin manifold? (b). Is the bordism class $[M,h]\in \Omega _*(BG)$ given by the holonomy map of an almost flat manifold trivial?
\vskip 2mm

For an almost flat  Spin  manifold $M$ with holonomy map $h: M\to
BG$, perhaps the $\eta$-invariant can be used to detect information on the
bordism class of $[M, h]$ (cf. \cite{GMP}.)

A strong conjecture of Farrell and Zdravkovska \cite{FZ} asked whether
an almost flat (or flat) manifold bounds a compact manifold whose interior admits a complete finite volume metric with negative (or constant negative) sectional curvature. Long and Reid \cite{LR} give counterexamples to the flat version of the conjecture using $\eta$-invariants.  As a corollary of the recent result of Ontaneda mentioned above, we see that this strong conjecture is true in the case of cyclic or quaternionic holonomy:

\begin{corollary}
 An almost flat manifold whose 2-sylow subgroup of the holonomy group is cyclic or quaternionic bounds a compact manifold whose interior admits a complete, finite volume Riemannian metric with negative pinched sectional curvature.
\end{corollary}

\section{Almost flat manifolds with cyclic or quaternionic holonomy bound}  \label{setup}

Recall that a {\em nilmanifold} is a quotient $N \backslash L$ of a simply connected nilpotent Lie group $L$ by a discrete cocompact subgroup $N$.   A nilmanifold is parallelizable; indeed one projects  a basis of left invariant vector fields on $L$ to the nilmanifold.  Since $N \subset L$ are both nilpotent, their centers $Z(N) \subset Z(L)$ are nontrivial.  Translation by an element of order two in $Z(N) \backslash Z(L)$ gives a fixed-point free involution on $N \backslash L$.

We now have two proofs that a nilmanifold bounds.  The first proof is that a nilmanifold is parallelizable, so its Stiefel-Whitney numbers are zero, hence by Thom's theorem it bounds.   The second proof is that a nilmanifold admits a fixed-point free involution, and any closed manifold with a fixed-point free involution  $\tau : M \to M$ is a boundary:
$$
\partial \left(  \frac{M \times [0,1]}{(m,t) \sim (\tau (m), 1-t)}\right) = M.
$$
Here The proof of our main result involves a combination of these two ideas.  Lemma \ref{+1 eigenbundle} is key.

 An {\em infranilmanifold} is a double coset space $M = \Gamma \backslash L \rtimes G/ G$ where $L$ is a simply connected nilpotent Lie group, $G$ is a finite subgroup of $\aut(L)$ and $\Gamma$ is a discrete torsion-free cocompact subgroup of $L \rtimes G$ which maps epimorphically to $G$ under the projection $L \rtimes G \to G$.  We require an infranilmanifold to have positive dimension.


  Let $N = \Gamma \cap L$.  Then $N$ is a normal subgroup of $\Gamma$ and the sequence
$$1\to N\to\Gamma \to G\to 1$$
is short exact.  Furthermore $N$ is a discrete cocompact subgroup of $L$ and hence is a finitely generated, torsion-free, nilpotent group.  The group $N$ is called the {\em nillattice}  and $G$ is called the {\em holonomy group} of the infranilmanifold.

We will replace the tautological
regular
covers
$$
L \rtimes G/G \to N \backslash L \rtimes G/G  \to \G \backslash L \rtimes G/G,
$$
using the diffeomorphisms $L \cong L \rtimes G/G$ and $N \backslash L \cong N \backslash L \rtimes G/G$.   We instead consider the covers
$$
L \xrightarrow{p} N \backslash L \xrightarrow{\pi} \G \backslash L \rtimes G/G .
$$
with $p(l) = Nl$ and $\pi(Nl) = \G (l,e) G$.   The affine group $\aff(L) = L \rtimes \aut L$ acts on $L$ via $(l',g)l = l'g(l)$.   Thus $\G < \aff(L)$ acts freely on $L$ and, in fact, $\pi \circ p$ is a regular $\G$-cover.   Likewise $G$ acts freely on $N \backslash L$ via $g(Nl) = N \gamma l$ where $\gamma \in \G$ maps to $g \in G$ and $\pi$ is a regular $G$-cover.   We call this $G$-action on $N \backslash L$ the {\em affine action}.

The $G$-action on $L$ given by $G < \aut L$ leaves the center $Z(L)$ invariant.   The $G$-action on $L$ can be  reinterpreted as conjugation in the affine group by using the short exact sequence $1 \to L \to L \rtimes G \to G \to 1$.  By comparing with the short exact sequence $1 \to N \to \G \to G \to 1$, we see that the $G$-action also leaves $Z(N)$ invariant but need not leave $N$ invariant.
We call the $G$-actions on $L$, $Z(L)$, $Z(N)$, and $Z(N) \backslash Z(L)$ {\em conjugation actions}.  They are all actions via group automorphisms.

\begin{definition}
 A {\em central involution} $\tau$ of an infranilmanifold $M = \G \backslash L \rtimes G/G$ is an element $\tau \in Z(N) \backslash Z(L)$ of order 2  which is invariant under the conjugation action of $G$.  This determines maps $\tau : M \to M, ~ \G(l,g)G \mapsto \Gamma(Tl,g) G$ and $\tau : N \bs L \to N \bs L, ~ N(l,g)G \mapsto N(Tl,g) G$ where $T \in Z(L)$ is a representative for $\tau$.
\end{definition}

%

Central involutions were key in all previous work on this problem \cite{HR}, \cite{FZ}, \cite{Upad} and will be for us too.

\begin{lemma}\label{2-group}

\begin{enumerate}

 \item Let $M$ be a closed manifold with an epimorphism of its  fundamental group to a finite group $G$.  Then $M$ bounds if and only if $H \backslash \widetilde M$ bounds where $\widetilde M$ is the $G$-cover of $M$ and $H$ is a 2-Sylow subgroup of $G$.

 \item  Any infranilmanifold $M = \G \backslash L \rtimes G/G$ with $G$ a 2-group has a central involution.

\end{enumerate}
\end{lemma}

\begin{proof}
 (i) Note that  $H \backslash \widetilde M \to M$ is a odd-degree cover, hence the domain and target have the same Stiefel-Whitney numbers; thus one bounds if and only if the other does.

  (ii)  Let $\Sigma$ be the subgroup of $Z(N)\backslash Z(L)$ generated by the elements of order 2.    Since the 2-group $G$ acts as group automorphisms of the abelian 2-group $\Sigma$, there must be at least two orbits of cardinality one, hence there must be a non-trivial element fixed by $G$.
\end{proof}

%
%
To analyze the fixed point set of a central involution, we  need a group theoretic remark.

\begin{remark} \label{two_actions}
Let $X$ be an $(H,K)$-biset where both $H$ and $K$ act freely on $X$.
Let $q : X \to  H \backslash X$ be the quotient map,
let $F$ be the fixed set of $K$ acting on $H \backslash X$, and let $\widetilde F = q^{-1}F$.  For $x \in \widetilde F$, there is a function $\varphi_x : K \to H$ so that $xk =  \varphi_x(k) x$.   Since $H$ acts freely, this function is uniquely defined, since $X$ is a biset, this function is a homomorphism,
and since $K$ acts freely, it is a monomorphism.
For a monomorphism $\varphi: K \to H$, let
$$\widetilde F_\varphi = \{x \in X \mid \forall k \in K, xk =  \varphi(k) x \}.
$$
Then $\widetilde F = \coprod \widetilde F_\varphi$.   It is easy to see that $h \widetilde F_\varphi = \widetilde F_{c_h \circ \varphi}$ where $c_h : H \to H$ is  conjugation.  Thus $q(\widetilde F_\varphi) = q(\widetilde F_{c_h \circ \varphi})$.
The group $H$ acts on the set of monomorphisms $K \to H$ by conjugation, let $[\varphi]$ denote an orbit. Let $F_{[\varphi]} = q(\widetilde F_\varphi)$.   Then $F = \coprod F_{[\varphi]}$.
\end{remark}

Note that there is a bijection between $(H,K)$-bisets and left $(H \times K)$-sets, where $hxk$ corresponds to   $hk^{-1}x$.

Later we will apply Remark \ref{two_actions} to analyze the fixed point set of a central involution $\tau : M \to M$ by setting $X = N \backslash L$, $H = G$ and $K = \langle \tau \rangle$.

A vector bundle $E \to B$ is {\em flat} if it has finite structure group, that is, $E \cong {\widetilde B} \times_G V$ for some finite, regular $G$-cover ${\widetilde B} \to B$ and some $\R G$-module $V$.  We call such a flat structure a {\em $(G,V)$-structure}.   Such a bundle (over a CW-complex) is the pullback of the flat bundle $EG \times_G V \to BG$ along a map $B \to BG$.   The regular $G$-cover $\wt B \to B$ can also be specified by a homotopy class of map $B \to BG$ or by a $G$-conjugacy class of homomorphism $\pi_1 B \to G$.


Tangent bundles of infranilmanifolds are flat:

\begin{lemma} \label{infranilsub_flat}
Consider an infranilmanifold $M = \G\bs L \rtimes G /G$.   Note that $T_eL$ is an $\R G$-module.
The tangent bundle of $M$ is flat:
$$
TM = (N \backslash L ) \times_G T_eL
$$
\end{lemma}

\begin{proof}
Note that $L$, like all Lie groups, is parallelizable.  Indeed there is an isomorphism of vector bundles over $L$ with $\Phi : L \times T_eL \xrightarrow{\cong} TL$ given by $\Phi(l,v) = d(L_l)_{e}v$, where $L_l : L \to L$ is left translation, defined by $L_l(l') = ll'$.
Note $L \rtimes G$ acts on $L$ via $(l',g)l = l'g(l)$, on $TL$ via the differential of this action, and on $T_eL$ via $v \mapsto (dg)_ev$.  It is straightforward to check that $\Phi$ is $L \rtimes G$-equivariant with respect to the diagonal $L \rtimes G$-action on the domain and the action given by the differential on the target.
 Verify the $L$-equivariance and $G$-equivariance separately and use the identity $g \circ L_l = L_{g(l)} \circ g : L \to L$.
 Note $\Gamma \subset L \rtimes G$ acts freely as deck transformations on  $L$ and the subgroup $N$ acts trivially on $T_eL$, so the map $ \Phi$ descends to the desired isomorphism $N \backslash L  \times_G T_eL \cong TM$.
 \end{proof}

\begin{lemma} \label{+1 eigenbundle}
Let $\tau: TM \to TM$ be a nontrivial bundle involution on the tangent bundle of a closed connected manifold $M$.   Let $F \subset M$ be the fixed set of the involution restricted to $M$.   If $\tau$ restricted to $TM|_F$ is the identity, then $M$ bounds.
\end{lemma}

\begin{proof}
Let $I = [0,1]$. Extend the involution on $M$ to an involution on $M \times I$ by setting $(m,t) \mapsto (\tau(m), 1-t)$.   Choose a closed involution invariant tubular neighborhood $N$ of the fixed set, thus $F \times \{1/2\} \subset N \subset M \times I$ (see Figure \ref{cylinder}).  Note that $(N, \partial N)$ is equivariantly diffeomorphic to the (disk bundle, sphere bundle) pair $(D(\nu \oplus 1), S(\nu \oplus 1))$ where $\nu = \nu(F \hookrightarrow M)$ is the normal bundle and where the involution on the bundle pair is given by fiberwise multiplication by $-1$.  Let $P(\nu \oplus 1)$ be the projective bundle $S(\nu \oplus 1)/(v \sim -v)$.  Then $W = \left(M\times I - \text{int }N \right)/\tau$ gives a cobordism of manifolds from $M$ to $P(\nu \oplus 1)$.

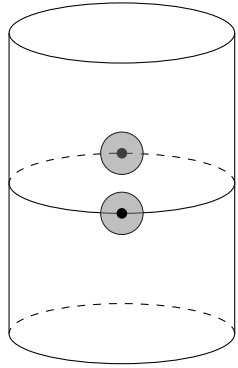
\begin{figure}
\begin{tikzpicture}
 \draw (0,4) arc (0:360: 1.5cm and .4cm ) ;
  \draw[dashed] (0,2) arc (0:180 : 1.5cm and .4cm ) ;
 \draw (0,2) arc (360:180: 1.5cm and .4cm ) ;
 \draw (-1.5cm,1.6cm) circle (8pt)   ;
 \fill [gray, opacity=.5]  (-1.5cm,1.6cm) circle (8pt)   ;
 \fill (-1.5cm,1.6cm) circle (2pt);
  \fill (-1.5cm,2.4cm) circle (2pt);
    \draw (-1.5cm,2.4cm) circle (8pt);
\fill [gray, opacity=.5]  (-1.5cm,2.4cm) circle (8pt);
  \draw[dashed] (0,0) arc (0:180 : 1.5cm and .4cm ) ;
 \draw (0,0) arc (360:180: 1.5cm and .4cm ) ;
 \draw (0,0) -- (0,4);
 \draw (-3,0) -- (-3,4);
\end{tikzpicture}
\caption{$M\times I$ with $M = S^1$ and $\tau(z) = \overline z$.   The disks are the tubular neighborhood $N$. }
\label{cylinder}
\end{figure}

We now upgrade to a cobordism of bundles.  Note that $TM \times I$ is a bundle over $M \times I$.  The bundle involution on $TM$ extends to a bundle involution $TM \times I$ by setting $(v,t) \mapsto (\tau(v), 1-t)$.  This descends to a  bundle $\xi$ over $W$ which restricts to $TM$ over $M$.   We wish to identify $\xi|_{P(\nu \oplus 1)}$.
By homotopy invariance of pullbacks of vector bundles (see Proposition 1.3 of \cite{Segal}), there is an equivariant isomorphism of vector bundles $TM \times I |_{D(\nu \oplus 1)} \cong \pi_D^*TM|_F$ where $\pi_D : D(\nu \oplus 1) \to F$ is the  bundle projection.   Since the involution on $TM|_F$ is trivial, there is an induced isomorphism  $\xi |_{P(\nu \oplus 1)} \cong \pi^*_P~ TM|_F$ where $\pi_P : P(\nu \oplus 1) \to F$ is the  bundle projection.

Thus $\xi$ gives a cobordism of bundles  from $TM$  to $\pi^*_P~ TM|_F$.   Since Stiefel-Whitney numbers are cobordism invariants, for all partitions $J$ of $\dim M$,
\begin{align*}
 w_J(TM)[M] &= w_J( \pi^*_P~ TM|_F  )[P(\nu \oplus 1)] \\
 & = (\pi^*_P w_J( TM|_F))[P(\nu \oplus 1)]\\
 & =0
\end{align*}
since $w_J (TM|_F)=0$ because $|J| = \dim M > \dim F$.   Thus all Stiefel-Whitney numbers of $M$ vanish, so by Thom's theorem $M$ bounds.

\end{proof}

\begin{remark}
One wonders if there could be a direct proof of Lemma \ref{+1 eigenbundle} which avoids the use of Thom's Theorem.   


Note that any involution $\tau: M \to M$ induces the involution $d \tau : TM \to TM$, but this involution does not restrict to the identity on $TM|_F$. 
\end{remark}

\begin{theorem}
An almost flat manifold bounds provided that the 2-sylow subgroup of the holonomy group is cyclic  or generalized quaternionic.
\end{theorem}

\begin{proof}
Let $M = \Gamma \backslash L \rtimes G/ G$  be an infranilmanifold with the 2-sylow subgroup of $G$ cyclic  or generalized quaternionic.
 By Lemma \ref{2-group}(i), we can pass to odd degree cover and assume that $G$ is a  2-group.
 Then there is a unique element $g \in G$ of order 2.  (In fact, according to a theorem of Burnside, a 2-group has a unique element of order 2 if and only if it is cyclic  or generalized quaternionic.)  Since the center of a $p$-group is nontrivial, $g$ is central.

 By Lemma \ref{2-group}(ii), there is a central involution $\tau \in Z(N)\backslash Z(L)$.   Since it is $G$-invariant, it induces a nontrivial involution $\tau : M \to M$. By Remark \ref{two_actions},  applied with $X = N \backslash L$, $H = G$, and $K = \langle \tau \rangle$, the fixed set of $\tau$ is
 $$F = \{[x] \in M \mid x \in N\backslash L, \tau x = g x\}.$$
 The involution $\tau$ on $M$ extends to the flat  tangent bundle $TM = (N \backslash L) \times_{G} T_eL$   via $\tau[x,v] = [\tau x, gv]$.   Then $\tau$ restricted to $TM|_F$ is the identity since for $[x] \in F$, $\tau[x,v] = [\tau x, gv] = [gx, gv] = [x,v]$.   Thus by Lemma \ref{+1 eigenbundle} $M$ is a boundary.
\end{proof}

The question of whether an infranilmanifold with holonomy group the Klein 4-group bounds remains open.

\vskip2mm

\bibliographystyle{amsalpha}

\vspace{15mm}

\end{document}